\documentclass[12pt,reqno]{amsart}
\usepackage{amsmath,amsthm,amssymb,tikz,graphicx,ifthen,enumerate,verbatim}
\usepackage[colorlinks,linkcolor=blue,citecolor=red]{hyperref}
\usepackage{float} 

\usepackage[parfill]{parskip}
\usepackage{stmaryrd}
\usepackage{rotating}

\usepackage{pifont}
\def\commentcolor{blue}
\newcounter{commentcounter} 
\newcommand{\COMMENT}[1]{\stepcounter{commentcounter}\ifnum \thecommentcounter=212\setcounter{commentcounter}{172}\fi{\normalsize\textcolor{\commentcolor} {\ensuremath{\smash{{}^\text{\footnotesize\ding{\thecommentcounter}}\!\!}}}}\marginpar{\color{\commentcolor}\footnotesize\vskip-.6 \baselineskip\noindent\raggedright\hsize1in{\normalsize\ding{\thecommentcounter}} \thinspace#1\vskip.3\baselineskip}%
\ifthenelse{\thecommentcounter=191}{\setcounter{commentcounter}{171}}{}%
\ifthenelse{\thecommentcounter=181}{\setcounter{commentcounter}{181}\def\commentcolor{magenta}}{}}

\newcommand{\R}{\mathbb{R}}
\newcommand{\C}{\mathbb{C}}
\newcommand{\D}{\mathbb{D}}
\newcommand{\Z}{\mathbb{Z}}
\providecommand{\Zi}{{\mathbb{Z}[i]}} \renewcommand{\Zi}{\mathcal{G}}
\newcommand{\CCnoD}{\mathbb{C}^2\!\setminus\!\Delta}
\newcommand{\rot}{\ensuremath{\iota}}
\newcommand{\dflip}{\ensuremath{\rho}}
\newcommand{\hflip}{\ensuremath{\eta}}
\renewcommand{\Re}{\mathrm{Re}}
\renewcommand{\Im}{\mathrm{Im}}
\newcommand{\Dih}{\mathrm{Dih}_4}
\newcommand{\NE}{_{\diamond}} 

\newcommand{\abs}[1]{\left\lvert\, #1 \,\right\rvert}

\newcommand{\cf}[1]{\left[\, #1 \,\right]}

\newcommand{\setform}[2]{\left\{\, #1 : #2 \,\right\}}

\newcommand{\<}[1][]{\begin{equation}\label{#1}}
\renewcommand{\>}{\end{equation}}
\newenvironment{cz}[1][]{\begin{center}\begin{tikzpicture}[#1]}{\end{tikzpicture}\end{center}}
\newcommand{\textdef}[1]{\textit{#1}}

\colorlet{gold}{yellow!80!black}
\colorlet{purp}{red!50!blue} 
\newcommand{\PB}[1]{ \fill [purp] (#1) circle (1); }
\def\sval{0.707107}
\def\Inf{2}

\newcounter{thmnum}
\newcommand{\startTheorem}[1]{\begin{itemize} \item[] \textbf{#1}\it}
\newenvironment{thm}{\refstepcounter{thmnum}\startTheorem{Theorem~\thethmnum.}}{\end{itemize}}
\newenvironment{lem}{\refstepcounter{thmnum}\startTheorem{Lemma~\thethmnum.}}{\end{itemize}}
\newenvironment{remark}[1][Remark.]{\begin{itemize} \item[] \textbf{#1} }{\end{itemize}}

\title{Bijectivity and trapping regions for complex continued fraction transformation}
\author{Adam Zydney}
\thanks{I would like to thank Svetlana Katok for suggesting this area of research and for many helpful conversations over the years.}
\date{May 1, 2016}

\makeatletter
\@ifclassloaded{amsart}{\title[Bijectivity region for complex CF transformation]{Bijectivity and trapping regions for complex continued fraction transformation}}{}
\makeatother

\begin{document}
\maketitle
\begin{abstract}
This paper provides some preliminary results on the dynamics of certain complex continued fractions. After establishing some general number theoretic results, we explore the dynamics of the natural extension map associated to a specific complex continued fraction algorithm (the ``diamond'' algorithm). We prove that this map has a bijectivity domain that is a subset of a trapping region for the map and, moreover, that both these sets have a ``finite product structure'' arising from a finite partition specific to the particular algorithm.
\end{abstract}

\vspace{1em}
Contents 

\tableofcontents
\hspace*{-4em}\raisebox{-2mm}{\color{white}\rule{4em}{1.2em}} 

\vspace{-1em}

\section{Introduction}\label{sec Intro}

Number theoretic properties of complex continued fractions were studied classically by Hurwitz \cite{Hurwitz87} and Khinchin \cite{Khinchin} and more recently by Doug Hensley \cite{Hensley} and S.~G.~Dani and Arnaldo Nogueira \cite{DN}.
Dynamical properties of real continued fractions, namely their connection to coding geodesics on the modular surface, go back to Artin \cite{Artin}, with further development by Caroline Series \cite{S80} and Adler and Flatto \cite{AF}. Katok and Ugarcovici \cite{KU05, KU10, KU12} describe a two-parameter family of minus continued fraction algorithms, which they call ``$(a,b)$-continued fractions.'' They describe some number theoretic properties, the dynamics of the associated natural extension maps, and applications to coding geodesics.

The main result of \cite{KU10} is that, with a few exceptions, the natural \mbox{extension} map $F_{a,b}$ on $\R^2$ has a global attractor set consisting of two connected components with ``finite rectangular structure,'' i.e., bounded by non-decreasing step functions with a finite number of steps.
The goal in this paper is to reach a similar result for the natural extension map associated to a particular complex continued fraction algorithm.

Section~\ref{sec NT} relates several properties of minus complex continued fractions, most of which are clear analogues of results in \cite{DN} for plus complex continued fractions.
Section~\ref{sec Diamond} gives a definition of ``finite product structure'' for sets in $\C^2$ and discusses the ``diamond algorithm'' on which the remainder of the paper is focused.
The two main results in Sections~\ref{sec Bij} and \ref{sec Trapping} are the existence and explicit description of a bijectivity domain $D\NE$ (Theorem~\ref{MAIN}) a trapping region $\Psi\NE \supset D\NE$ (Theorem~\ref{PSI}), both of which have a finite product structure. The proof that $\Psi\NE$ traps points (Lemma~\ref{V lemma}) depends on complex continued fraction theory, namely Theorem~\ref{complex convergence}.


\section{Complex continued fractions}\label{sec NT}

A \textdef{continued fraction}, or \textdef{c.f.} for short, is any expression of the form
\[ \label{cf-in-general}
	a_0 + \dfrac{b_1}{
	a_1+\dfrac{b_2}{a_2+\dfrac{b_3}{\ddots
	\raisebox{-.4em}{$+\;\tfrac{b_n}{a_n}$}
	}}}
	\qquad\text{or}\qquad
	a_0 + \dfrac{b_1}{
	a_1+\dfrac{b_2}{a_2+\dfrac{b_3}{\ddots
	}}}.
\]
A \textdef{minus continued fraction} is one in which $b_n = -1$ for all $n$, and a \textdef{plus continued fraction} has all $b_n = +1$.
For the most part, this paper will deal only with infinite minus continued fractions. 

Given the two sequences $\{a_n\}$ and $\{b_n\}$, one can define sequences $\{p_n\}$ and $\{q_n\}$ by
\[ 
	\begin{array}{l@{\qquad\qquad}l@{\qquad\qquad}l}
    	p_{-2} = 0    &   p_{-1} = 1   &   p_n = a_n p_{n-1} + b_n p_{n-2}
    	\quad\text{for } n \ge 0; \\
    	q_{-2} = -1   &   q_{-1} = 0   &   q_n = a_n q_{n-1} + b_n q_{n-2}
    	\quad\,\text{for } n \ge 0.
	\end{array}
\]
Algebraic manipulations show that for all $n \ge 0$, 
\[
	\dfrac{p_n}{q_n} = 
	a_0 + \dfrac{b_1}{
	a_1+\dfrac{b_2}{a_2+\dfrac{b_3}{\ddots
	\raisebox{-.4em}{$+\;\tfrac{b_n}{a_n}$}
	}}}
\]
assuming $a_n \ne 0$.
This holds for $a_k$ and $b_k$ in any ring or field, not necessarily $\R$ or $\C$. The fraction $\frac{p_n}{q_n}$ is called the $n^\text{th}$ \textdef{convergent} of the continued fraction.

Since we deal only with minus continued fractions from now on, we introduce the notations
\[
	\cf{a_0, a_1, a_2, \ldots, a_n} =
	a_0 - \dfrac{1}{
	a_1-\dfrac{1}{a_2-\dfrac{1}{\ddots
	\raisebox{-.4em}{$-\;\tfrac{1}{a_n}$}
	}}}
\]
and 
\[
	\cf{a_0, a_1, a_2, \ldots} = 
	a_0 - \dfrac{1}{
	a_1-\dfrac{1}{a_2-\dfrac{1}{\ddots
	}}}.
\]
Occasionally in proofs, these notations may be used with non-integer $a_n$, and  the notation $\cf{a_0, a_1, a_2, \ldots}$ may be used as a formal expression even if the sequence $r_n = \cf{a_0, a_1, \ldots, a_n}$ has no limit as $n \to \infty$.

We also simplify the recurrence relations above to 
\< \label{pq recurrence}
	\begin{array}{l@{\hspace{3em}}l@{\hspace{3em}}l}
    	p_{-2} = 0    &   p_{-1} = 1   &   p_n = a_n p_{n-1} - p_{n-2}
    	\quad\text{for } n \ge 0; \\
    	q_{-2} = -1   &   q_{-1} = 0   &   q_n = a_n q_{n-1} - q_{n-2}
    	\quad\,\text{for } n \ge 0.
	\end{array}
\>

Denote the set of Gaussian integers by
\[ \Zi = \setform{x+yi}{x,y \in \Z}. \]
For the remainder of this paper, elements of $\Zi$ may also be referred to as complex integers or simply integers. Additionally, a \textdef{rational} complex number is an element of $\mathbb Q(i)$, and therefore an \textdef{irrational} complex number is one for which the real or imaginary parts or both are irrational.

Plus complex continued fractions have been studied by Adolf Hurwitz~\cite{Hurwitz87}, Doug Hensley~\cite{Hensley}, and more recently by S.\,G.\,Dani and Arnaldo Nogueira~\cite{DN}, who introduce the terms ``choice function'' and ``iteration sequence.''\footnote{\,The sequence $\{z_n\}$ defined in (\ref{complex cf algorithm}) is an example of an iteration sequence. Dani and Noguiera give conditions for an arbitrary sequence $\{z_n\}$ to be useful for constructing continued fractions, but this paper does not deal with such general iteration sequences.} A \textdef{choice function} is a function $c:\C \to \Zi$ for which $c(0) = 0$ and $\abs{z - c(z)} \le 1$ for all $z$ (that is, $c(z)$ chooses a Gaussian integer that is at most a distance of 1 from $z$). For any given choice function, one can define the set 
\[ \Phi_c = \overline{\setform{z - c(z)}{z \in \C}}. \]
The most classical example is the \textdef{Hurwitz} or \textdef{nearest integer} choice function which maps each complex number to its nearest Gaussian integer (with some convention for points equidistant from multiple nearest Gaussian integers). This algorithm was discussed in detail by Hurwitz~\cite{Hurwitz87}, and in this case the set 
\[ \Phi_\text{Hurwitz} = \setform{ x+yi }{ \abs{x}\le \tfrac12 , \abs{y}\le \tfrac12 } \]
is a unit square centered at the origin.

\begin{remark}
The definition of a choice function implies that $\Phi_c \subset \overline{B(0,1)}$. In many cases, such as Hurwitz, it also also true that $\Phi_c \subset B(0,1)$ or even that $\Phi_c \subset B(0,r)$ for some $r < 1$. There are some number theoretic results that require this additional restriction on $c$, but many do not.
\end{remark}

\pagebreak[1]
One can also construct a choice function starting with a valid set $\Phi$.
\vspace{-1em}

\begin{lem}\label{Phi to c}
	Let $\Phi \subset \overline{B(0,1)}$ contain $0$ and let
	\[ f_\emptyset(z) = \left\{\begin{array}{ll}
		z-1 & \text{if } -\pi/4 \le \arg z < \pi/4 \\
		z-i & \text{if } \pi/4 \le \arg z < 3\pi/4 \\
		z+1 & \text{if } 3\pi/4 \le \arg z \text{ or } \arg z < -3\pi/4 \\
		z+i & \text{if } -3\pi/4 \le \arg z < -\pi/4.
	\end{array}\right. \]
	If for any $z \notin \Phi$ there exists an integer $N_z\ge0$ such that $f_\emptyset^{N_z}(z) \in \Phi$, then the function 
	\[ c_\Phi(z) = \left\{\begin{array}{ll}
		0 & \text{if } z \in \Phi \\
		f_\emptyset^{N_z}(z) & \text{if } z \notin \Phi
	\end{array}\right. \]
	is a valid choice function.
\end{lem}
\begin{remark} The two processes of moving from $c$ to $\Phi$ and from $\Phi$ to $c$ are not inverses: for a choice function $h$, the choice function $c_{(\Phi_h)}$ may not be equal to $h$. \end{remark}


For a given choice function $c$, sequences $\{z_n\}$ and $\{a_n\}$ are defined~by
\< \label{complex cf algorithm}
	z_0 = z, \qquad 
	a_n = c(z_n), \qquad
	z_{n+1} = \frac{-1}{z_n-a_n} \quad\text{for }n\ge0.
\>

Dani and Noguiera~\cite{DN} deal exclusively with plus continued fractions, but the relevant statements can be re-stated and re-proved for minus c.f. For example, a version of Lemma~\ref{DN3.3i} below is stated for plus c.f.~as part of Proposition 3.3 in~\cite{DN}, where it has an additional $(-1)^{n+1}$ term on the right-hand side corresponding to the fact that \mbox{$p_{n+1}q_n - q_{n+1}p_n = (-1)^n$} for plus c.f.~but equals 1 for all $n \ge 0$ for minus c.f. 
\begin{lem} \label{DN3.3i}
	Let $\{a_n\}$ and $\{z_n\}$ be sequences satisfying $z_0 = z$ and $z_{n+1} = \frac{-1}{z_n-a_n}$. Define $\{p_n\}$ and $\{q_n\}$ by~(\ref{pq recurrence}).	
	 Then \[ p_n - q_n z = (z_1 \cdots z_{n+1})^{-1} \] for all $n \ge 0$.
\end{lem}

\textit{Proof by induction.}
	For $n=0$, we have $p_0 = a_0$, $q_0 = 1$, and $z_1 = \frac{-1}{z_0 - a_0}$, so by direct calculation $p_0 - q_0 z = z_1^{-1}$.
	Now let $n \ge 1$ and assume $p_k - q_k z = (z_1 \cdots z_{k+1})^{-1}$ for $k = 0, \ldots, n-1$. Then
	\begin{align*}
		p_n - q_n z
		&= (a_n p_{n-1} - p_{n-2}) - (a_n q_{n-1} - q_{n-2}) z \\
		&= a_n p_{n-1} - p_{n-2} - a_n q_{n-1} z + q_{n-2} z \\
		&= a_n (p_{n-1} - q_{n-1} z) - (q_{n-2} z - p_{n-2}) \\
		&= a_n (z_1 \cdots z_n)^{-1} - (z_1 \cdots z_{n-1})^{-1} \\
		&= (z_1 \cdots z_n)^{-1} (a_n - z_n) \\
		&= (z_1 \cdots z_n)^{-1} z_{n+1}^{-1} \\
		&= (z_1 \cdots z_{n+1})^{-1}
\tag*{$\square$}
\end{align*}

\begin{lem} \label{DN3.6}
	Under the setup of Lemma~\ref{DN3.3i}, $\abs{z_1 \cdots z_n} \to \infty$ as $n \to \infty$.
\end{lem}
See \cite[Prop.\,3.6]{DN} for a proof for plus continued fractions; this proof applies equally well to minus c.f.~when ``$z_{n+1} = \frac1{\beta_n + \zeta_n}$'' is replaced by ``$-z_{n+1} = \frac1{\beta_n + \zeta_n}$.''

%
%
%

\begin{thm} \label{complex convergence}
	Let $c:\C \to \Zi$ be a choice function such that \mbox{$\abs{z \!-\! c(z)} \ne 1$} if $z$ is irrational and $\abs z = 1$. Let $z \in \C$ be irrational, let $\{a_n\}$ be the sequence of defined by~(\ref{complex cf algorithm}) above, and let $\{p_n\}$ and $\{q_n\}$ be defined exactly as in~(\ref{pq recurrence}). Then $q_n \ne 0$ for all $n \ge 0$, the sequence $\{\frac{p_n}{q_n}\}$ converges to $z$, and $\abs{q_n} \to \infty$ as $n \to \infty$.
\end{thm}

\begin{proof}
	Note that $\abs{z_n} \ge 1$ for all $n \ge 1$ because $\abs{z_n - c(z_n)} \le 1$ by the definition of a choice function and $\abs{z_{n+1}} = \abs{z_n - c(z_n)}^{-1}$ by~(\ref{complex cf algorithm}). From this, we have that \mbox{$\abs{z_1 \cdots z_{n+1}} \ge 1$} for any $n \ge 1$, but in fact a slightly stronger statement is true.
	\[
		\abs{z_1}
		= \abs{\frac{-1}{z_0 - a_0}} \\
		= \abs{z_0 - a_0}^{-1}
		= \abs{z - c(z)}^{-1}
	\]
	Suppose $\abs{z - c(z)} < 1$. Then $\abs{z_1} > 1$, and $\abs{z_1 \cdots z_{n+1}} \ge 1$ can be strengthened to $\abs{z_1 \cdots z_{n+1}} > 1$ for any $n \ge 0$.
	
	Now suppose $\abs{z - c(z)} = 1$. Then $\abs{z_1} = 1$, which by the assumption of the theorem means $\abs{z_1 - c(z_1)} \ne 1$ and thus $\abs{z_1 - c(z_1)} < 1$. Then 
	
	\[
		\abs{z_2} 
		= \abs{\frac{-1}{z_1 - a_1}} 
		= \abs{z_1 - a_1}^{-1}
		= \abs{z_1 - c(z_1)}^{-1}
		> 1.
	\]
	Thus $\abs{z_1 \cdots z_{n+1}} > 1$ for any $n \ge 1$ no matter whether $\abs{z - c(z)} < 1$ or not.
	
	By Lemma~\ref{DN3.3i}, $p_n - q_n z = (z_1 \cdots z_{n+1})^{-1}$ for all $n \ge 1$ and thus $0 \le \abs{p_n - q_n z} < 1$ for all $n \ge 1$. However $\abs{p_n - q_n z} = 0$ is impossible because then $z = \frac{p_n}{q_n}$, contradicting the condition in the theorem that $z$ be irrational. Thus $0 < \abs{p_n - q_n z} < 1$ for all $n \ge 1$.
	Recall $q_0 = 1$. If $q_k = 0$ for some $k \ge 1$, then $\abs{p_k - q_k z}$ would equal just $\abs{p_k}$. Since $p_k \in \Zi$, it cannot be that $0 < \abs{p_k} < 1$; thus $q_n \ne 0$ for all $n \ge 0$.
	\begin{align*}
		\abs{\frac{p_n}{q_n}-z}
		&= \abs{\!\left(\frac{p_n}{q_n}-z\right)\!q_n} \abs{q_n}^{-1}
		\\&= \abs{p_n - q_n z} \abs{q_n}^{-1}
		= \abs{z_1 \cdots z_{n+1}}^{-1} \abs{q_n}^{-1}
	\end{align*}
	Knowing that $\abs{z_1 \cdots z_{n+1}}^{-1} \to 0$ from Lemma~\ref{DN3.6} and that $\abs{q_n}^{-1}$ is bounded (by~1 since \mbox{$q_n \in \Zi\!\setminus\!\{0\}$} and thus $\abs{q_n} \ge 1$), we have that $\big\lvert\,\frac{p_n}{q_n}-z\,\big\rvert \to 0$. Therefore the sequence $\{p_n/q_n\}$ converges to $z$.
	
	Lastly, assume 
	$\abs{q_n} \le M$. Then $1 \ge \abs{q_n}^{-1} \ge 1/M$ and 
	\[ \abs{p_n} \ge \frac{\abs{p_n}}{\abs{q_n}} \ge \frac{\abs{p_n}}M. \]
	If $\abs{p_n}$ diverges, then $\abs{p_n}/M$ diverges as well, but we know $\abs{p_n/q_n}$ converges to~$\abs z$. Thus $\abs{p_n}$ must converge. A converging sequence from a discrete set must be eventually constant. If $\abs{p_n}$ is constant for all $n > N$, then $\{p_n\}_{n\ge N}$ has only a finite set of values, and since $\abs{q_n}$ is bounded, $\{q_n\}_{n\ge N}$ also has only finitely many values. This means that the set $\setform{ p_n/q_n }{n\ge N}$ is finite. A converging sequence from a finite set must eventually equal its limit, so $z$ must equal exactly $p_n/q_n$ for some $n$.
	%
	%
	However, this contradicts the irrationality of $z$. For an irrational $z$, then, it must be that $\{q_n\}$ is not bounded and thus that $\abs{q_n} \to \infty$.
\end{proof}

\section{Diamond algorithm and its partition}\label{sec Diamond}

The remainder of this paper deals exclusively with the ``diamond algorithm,'' which uses the fundamental set
\<[Phi diamond] \Phi\NE := \setform{ x+yi }{ \abs{x} + \abs{y} \le 1 } \>
and the choice function $c\NE := c_{\Phi\NE}$ defined as described in Lemma~\ref{Phi to c}. This algorithm was first described by Julius Hurwitz in 1902 \cite{HurwitzJ}.

The three maps
\[ T(z) = z+1, \qquad U(z) = z+i, \qquad S(z) = \frac{-1}z \]
and their inverses are the basis for various transformations related to complex continued fractions, including the piecewise continuous map $f\NE:\C\to\C$ defined as
\< \label{fNE}
	f\NE(z) = \left\{\begin{array}{ll}
		Sz & \text{if }z \in \Phi\NE \\
		T^{-1}z & \text{if } -\pi/4 \le \arg z < \pi/4 \\
		U^{-1}z & \text{if } \pi/4 \le \arg z < 3\pi/4 \\
		Tz & \text{if } 3\pi/4 \le \arg z \text{ or } \arg z < -3\pi/4 \\
		Uz & \text{if } -3\pi/4 \le \arg z < -\pi/4.
	\end{array}\right.
\>
The ``pieces'' of this piecewise definition are designed to bring any point $z \notin \Phi\NE$ closer to the origin by integer translation until it enters the set $\Phi\NE$ and is inverted. Figure~\ref{fig diamond regions} shows the regions of $\C$ on which this function acts by different maps ($\Phi\NE$ is shaded in the figure). A set is called \textdef{consistent} if the map $f\NE$ acts on all points in the set by the same generator, meaning that the set is contained in only one of these regions.
\def\Inf{2}
\begin{figure}[h]
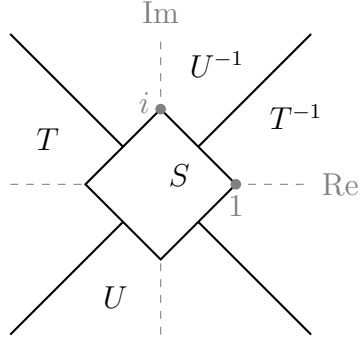

\begin{cz}
	\draw [gray,dashed] (-\Inf,0) -- (\Inf,0) node [right] {$\Re$};
	\draw [gray,dashed] (0,-\Inf) -- (0,\Inf) node [above] {$\Im$};
	\draw [thick] (-\Inf,-\Inf) -- (\Inf,\Inf);
	\draw [thick] (-\Inf,\Inf) -- (\Inf,-\Inf);
	\draw [fill=white,thick] (1,0) -- (0,1) -- (-1,0) -- (0,-1) -- cycle;
	\fill [gray] (1,0) circle (2pt) node [below] {1};
	\fill [gray] (0,1) circle (2pt) node [above=.2em,left] {$i$};
	
	\draw (1/4,1/8) node [fill=white] {$S$};
	\draw (1.8,0.9) node [fill=white] {$T^{-1}$};
	\draw (0.75,1.6) node [fill=white] {$U^{-1}$};
	\draw (-1.5,.6) node [fill=white] {$T$};
	\draw (-.6,-1.5) node [fill=white] {$U$};
\end{cz}
\caption{Action of $f\NE$ in different regions}
\label{fig diamond regions} 
\end{figure}

The natural extension map of $f\NE$ is the map $F\NE:\CCnoD\to\CCnoD$, where $\Delta = \setform{(z,w)\in\C^2}{z=w}$, given by
\< \label{FNE}
	F\NE(z,w) = \left\{\begin{array}{ll}
		(Sz,Sw) & \text{if }w \in \Phi\NE \\
		(T^{-1}z,T^{-1}w) & \text{if } \tfrac{-1}4\pi \le \arg w < \tfrac14\pi \\[0.15em]
		(U^{-1}z,U^{-1}w) & \text{if } \tfrac14\pi \le \arg w < \tfrac34\pi \\[0.25em]
		(Tz,Tw) & \text{if } \tfrac34\pi \le \arg w \text{ or } \arg w < \tfrac{-3}4\pi \\[0.2em]
		(Uz,Uw) & \text{if } \tfrac{-3}4\pi \le \arg w < -\tfrac14\pi
	\end{array}\right.
\>
and is the main object of study for the remainder of this paper.


Analogous to the ``finite rectangular structure'' attractor region for the real natural extension map $F_{a,b}$ described in~\cite{KU10,KU12}, the goal is to find an invariant set $D\NE$ for the map $F\NE$ that has \textdef{finite product structure}, \mbox{meaning it can be expressed as a finite union of Cartesian products:}
\< \label{DNE FPS} D\NE = \bigcup_{k=1}^N (Z_k \times W_k) \>
and each set $W_k$ is consistent.
This is accomplished by forming a finite partition of~$\C$ satisfying the following properties:
\begin{enumerate}[(i)]
	\item the set $\Phi\NE$ is a union of sets from this partition; \\[-1.4em]
	\item the image under $f\NE$ of any set in the partition is a union of sets from this partition. \\[-1.4em]
\end{enumerate}
The natural way to do this is to look at the all images of $\Phi\NE$ under $f\NE$ and take a partition fine enough to describe every one of these images as a union of partition elements
. A priori, there is no reason to think that such a partition will be finite or have any nice presentation, and indeed it is unclear what can be said for a generic choice function or fundamental set. In the specific case of $f\NE$, however, this partition is quite nice (see Figure~\ref{fig C partition}).

Property (\ref{partition image prop}) above implies that the image under $F\NE$ of any set with finite product structure must also have finite product structure.
\begin{lem}\label{finiteness}
	Let $W_1,\ldots,W_N$ be a collection of sets such that for each $k$ there exists a set of indices $J_k \subset \{1,\ldots,N\}$ satisfying $f\NE(W_k) = \bigcup_{j\in J_k} W_j$. Then the image under $F\NE$ of any set with finite product structure must also have finite product structure.
\end{lem} 
\begin{proof}
For each $k$, let $h_k$ be one of the maps $T,T^{-1},U,U^{-1},S$ chosen so that $f\NE$ acts on $W_k$ by the map $h_k$. Thus $f\NE(W_k) = h_k(W_k)$ and moreover $F\NE(Z_k \times W_k) = h_k(Z_k) \times h_k(W_k)$.

Let $A = \bigcup_{\,k=1}^{\,N} (Z_k \times W_k)$ be some set with finite product structure.
\begin{align*}
	F\NE(A) &= F\NE\Big( \bigcup_{1\le k\le N} (Z_k \times W_k) \Big) \\[-.25em]
	&= \bigcup_{1\le k\le N} F\NE(Z_k \times W_k) \\[-.15em]
	&= \bigcup_{1\le k\le N} h_k(Z_k) \times h_k(W_k) \\[-.15em]
	&= \bigcup_{1\le k\le N} \Big( h_k(Z_k) \times \bigcup_{j\in J_k} W_j \Big) \\[-.15em]
	&= \bigcup_{1\le k\le N} \bigcup_{j\in J_k} (h_kZ_k \times W_j)
\end{align*}
Since each $J_k$ is finite, the double-union over $k\in\{1,\ldots,N\}$ and $j \in J_k$ is still a finite union.
\end{proof}

From the eventual construction, it will be seen that equation~(\ref{DNE FPS}) can be satisfied with $N=40$ sets, but a more compact description can be given using symmetry. The set $\Phi\NE$ is a diamond, which has symmetry group $\Dih$ (the dihedral group of degree 4 and order 8), and it will be convenient to do calculations with in the quotient space $\C/\Dih$, which is naturally identified with the ``wedge-shaped'' set of points $\setform{w\in\C\!}{\!0\le\arg w<\pi/4}$, which we will denote by $\C^*$.

Define the following sets, which can be seen in Figure~\ref{fig C partition}.
\begin{align}
	W_1 &= \hspace{.17em}\setform{ w \in \C^* }{ \Im\,w \le \Re\,w - 1 } \notag \\
	W_2 &= \setform{ w \in \C^* }{ \Im\,w \ge \Re\,w - 1, \abs{w-(\tfrac12+\tfrac12i)} \ge \tfrac1{\sqrt2} } \notag \\
	W_3 &= \setform{ w \in \C^* }{ \Im\,w \ge 1 - \Re\,w, \abs{w-(\tfrac12+\tfrac12i)} \le \tfrac1{\sqrt2} } \label{Wk} \\
	W_4 &= \setform{ w \in \C^* }{ \Im\,w \le 1 - \Re\,w, \abs{w-(\tfrac12-\tfrac12i)} \ge \tfrac1{\sqrt2} } \notag \\
	W_5 &= \setform{ w \in \C^* }{ \abs{w-(\tfrac12-\tfrac12i)} \le \tfrac1{\sqrt2} } \notag
\end{align}

\begin{figure}[h]
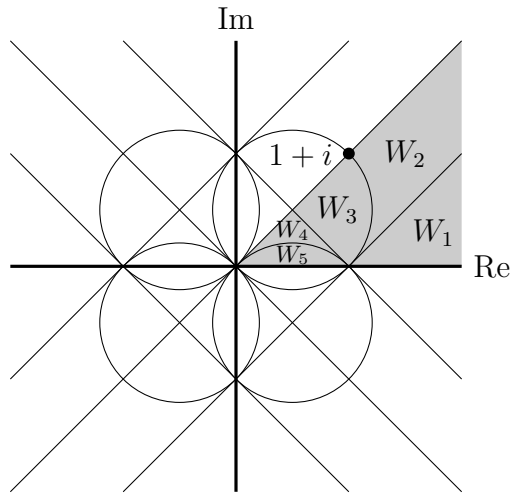

\begin{cz}[scale=1.5]
	\fill [black!20] (\Inf,\Inf) -- (0,0) -- (\Inf,0);

	\draw (-\Inf,-\Inf) -- (-1/2,-1/2);
	\draw (1/2,1/2) -- (\Inf,\Inf);
	\draw (-\Inf,\Inf) -- (-1/2,1/2);
	\draw (1/2,-1/2) -- (\Inf,-\Inf);
	\draw (-1,0) -- (0,1) -- (1,0) -- (0,-1) -- cycle;
	\draw [very thick] (-\Inf,0) node [left,white] {$\Re$} -- (\Inf,0) node [right] {$\Re$}; 
	\draw [very thick] (0,-\Inf) -- (0,\Inf) node [above] {$\Im$};;
	
	\foreach \r in {0,90,180,270} {
		\draw [rotate=\r] (1/2,1/2) circle (\sval);
		\draw [rotate=\r] (\Inf,\Inf-1) -- (1,0) -- (0,0) -- (0,1) -- (\Inf-1,\Inf);
		\draw [rotate=\r] (0,0) -- (1/2,1/2);
	}
	
	\draw (1.75,.3) node {$W_1$};
	\draw (1.5,1) node {$W_2$};
	\draw (0.9,.5) node {$W_3$};
	\draw (0.5,.32) node {\scriptsize$W_4$};
	\draw (0.5,.1) node {\scriptsize$W_5$};
	
	\fill (1,1) circle (1.5pt) node [below=0.08em,left] {$1+i$\,};
\end{cz}
\caption{Partition of $\C$ with $\C^*$ in gray}
\label{fig C partition} 
\end{figure}

The sets $W_1,\ldots,W_5$ partition the wedge $\setform{w\in\C\!}{\!0\le\arg w<\pi/4}$ (ignoring overlap on their boundaries), and thus the collection of sets $\xi W_k$ with $\xi \in \Dih$ and $k=1,2,3,4,5$ form a partition of $\C$ itself, shown in Figure~\ref{fig C partition}. Since $\abs{\Dih} = 8$, there are $8 \times 5 = 40$ total sets in this partition, but all relevant calculations can be carried out using only the~5 ``standard'' sets given above.

\pagebreak[10]
\begin{lem} \label{partition finiteness}
	The partition $\setform{ \xi W_k }{ \xi \in \Dih, 1 \le k \le 5 }$ satisfies the following: 
	\begin{enumerate}[(i)]
		\item\label{partition Phi prop} the set $\Phi\NE$ is a union of sets from this partition;
		\item\label{partition image prop} the image under $f\NE$ of any set in the partition is a union of sets from this partition.
	\end{enumerate}
\end{lem}
\begin{proof}
	The set $W_4 \cup W_5$ is the projection of $\Phi\NE$ to $\C/\Dih$, that is,
	\[ \Phi\NE = \bigcup_{\xi\in\Dih}\!\xi W_4 \;\cup \bigcup_{\xi\in\Dih}\!\xi W_5, \]
	which proves (\ref{partition Phi prop}).
	
	To prove (\ref{partition image prop}), it is sufficient to express the images of $W_1, \ldots, W_5$ as unions of partition elements:
	\<[fNE Wk]
	\begin{array}{r@{\;=\;}l@{\;=\;}l}
		f\NE(W_1) & T^{-1}W_1 & W_1 \cup W_2 \cup W_3 \cup W_4 \cup W_5 \\
		f\NE(W_2) & T^{-1}W_2 & \dflip W_2 \cup \dflip W_3 \cup \dflip W_4 \\
		f\NE(W_3) & T^{-1}W_3 & \dflip W_5 \cup \rot W_5 \cup \rot W_4 \\
		f\NE(W_4) & SW_4 & \hflip W_2 \\
		f\NE(W_5) & SW_5 & \hflip W_1,
	\end{array} \>
	where $\dflip: w \mapsto i\,\overline{w}$ (reflection across $\Im\,w = \Re\,w$), $\rot: w \mapsto i\,w$ (counter\-clockwise rotation by 90$^\circ$), and $\hflip : w \mapsto -\overline{w}$ (reflection across the $\Im$-axis) are all elements of $\Dih$.
\end{proof}

\section{Bijectivity domain}\label{sec Bij}

We can now use the sets $W_k$ in the partition to describe a bijectivity domain for $F\NE$.

\begin{thm} \label{MAIN}
	There exists a set $D\NE \subset \CCnoD$ such that
	\begin{enumerate}[(i)]
		\item\label{MAIN bij prop} $D\NE$ is a \textdef{bijectivity domain}, meaning that $F\NE(D\NE) = D\NE$ and the map $F\NE : D\NE \to D\NE$ is bijective except on parts of the boundary of $D\NE$;
		\item\label{MAIN fps prop} $D\NE$ has finite product structure: there exists a finite collections of sets $Z_1,W_1,\ldots,Z_N,W_N \subset \C$ such that
		\[ \tag{\ref{DNE FPS}} D\NE = \bigcup_{k=1}^N (Z_k \times W_k); \]
		\item\label{MAIN set prop} each set $Z_k$ and $W_k$ in~(\ref{DNE FPS}) can be given explicitly, and each is a connected set whose boundary consists of straight lines (infinite or segments) and arcs of circles.
	\end{enumerate}
\end{thm}
\begin{remark}
	Ideally, $D\NE$ should actually be an attractor region, which would require that the orbit of almost any point $(z,w)$ in $\CCnoD$ enters $D\NE$ in finite time or at least asymptotically. This has yet to be proved or disproved. Section~\ref{sec Trapping} describes a set $\Psi\NE \supset D\NE$ with this property.
\end{remark}

The proof of Theorem~\ref{MAIN} consists of descriptions of the sets $Z_k$ and $W_k$ satisfying~(\ref{MAIN set prop}), defining the set $D\NE$ as in~(\ref{MAIN fps prop}), and calculations to prove~(\ref{MAIN bij prop}). This information is presented below along with the method used to construct and determine these sets; see the left half of Figure~\ref{fig DNE} for a visualization of $D\NE$.


The sets $Z_1,\ldots,Z_5$ for equation (\ref{DNE FPS}) are the following:
\begin{align}
	Z_1 &= \setform{ z \in \C }{ \Re\,z \le \tfrac12, \abs{\Im\,z} \le \tfrac12, \abs{z-1} \ge 1 } \notag \\
	Z_2 &= \setform{ z \in \C }{ \Re\,z \le \tfrac12,\, \Im\,z \le \tfrac12, \abs{z-1} \ge 1 } \notag \\
	Z_2 &= \setform{ z \in \C }{ \Re\,z \le \tfrac12,\, \Im\,z \le \tfrac12, \abs{z} \ge 1 } \label{Zk} \\
	Z_4 &= \setform{ z \in \C }{ \Re\,z \le \tfrac12, \abs{z} \ge 1, \abs{z-i} \ge 1 } \notag \\
	Z_5 &= \setform{ z \in \C }{ \Re\,z \le \tfrac12, \abs{z} \ge 1, \abs{z-i} \ge 1, \abs{z+i} \ge 1 } \notag
\end{align}
Using $W_1,\ldots,W_5$ as defined in (\ref{Wk}) and $Z_1,\ldots,Z_5$ as defined in (\ref{Zk}), we define the set $D\NE$ as 
\< \label{DNE FPS Dih} D\NE = \bigcup_{k=1}^5 \bigcup_{\xi \in \Dih} \xi(Z_k \times W_k), \>
where $\Dih$ acts on $\C^2$ by $\xi(z,w) = (\xi z,\xi w)$. Equation~(\ref{DNE FPS Dih}) is a restatement of (\ref{DNE FPS}) with $Z_6,\ldots,Z_{40}$ each being an image of one of $Z_1,\ldots,Z_5$ under an element of $\Dih$ and likewise for $W_5,\ldots,W_{40}$. The set $D\NE$ can most easily be visualized using the five products on the left of Figure~\ref{fig DNE}. \textbf{Note that each~product $\boldsymbol{Z\!\times\!W}$ shown in Figure~\ref{fig DNE} actually represents $\boldsymbol{\bigcup_{\xi\in\Dih}\xi(Z\!\times\!W)}$.}

\newcommand{\Product}[4][scale=1.28]{ 
\begin{tikzpicture}[#1]
	\def\Inf{3}
	\begin{scope}[scale=1/\Inf]
		\begin{scope} \clip (-\Inf,-\Inf) rectangle (\Inf,\Inf); #3 \end{scope}
		
		\draw [black, thick] (-\Inf,0) -- (\Inf,0);
		\draw [black, thick] (0,-\Inf) -- (0,\Inf);
		 
		\draw [black,opacity=.25] (-\Inf,-\Inf) grid (\Inf,\Inf);
		\draw (-\Inf,-\Inf) rectangle (\Inf,\Inf);
	\end{scope}
		
	\draw (1.2,0) node {$\times$};
	
	\begin{scope}[xshift=2.4cm]
	\begin{scope}[scale=1/\Inf]
		\begin{scope} \clip (-\Inf,-\Inf) rectangle (\Inf,\Inf); #4 \end{scope}
		
		\draw [black, thick] (-\Inf,0) -- (\Inf,0);
		\draw [black, thick] (0,-\Inf) -- (0,\Inf);
		 
		\draw [black,opacity=.25] (-\Inf,-\Inf) grid (\Inf,\Inf);
		\draw [dashed] (-\Inf,-\Inf) -- (-1/2,-1/2);
		\draw [dashed] (1/2,1/2) -- (\Inf,\Inf);
		\draw [dashed] (-\Inf,\Inf) -- (-1/2,1/2);
		\draw [dashed] (1/2,-1/2) -- (\Inf,-\Inf);
		\draw [dashed] (-1,0) -- (0,1) -- (1,0) -- (0,-1) -- cycle;
		\draw (-\Inf,-\Inf) rectangle (\Inf,\Inf);
	\end{scope}
	\end{scope}
	
	\draw (1.25,-1.2) node [below] {\smash{$#2$}};
\end{tikzpicture}\vspace{-0.5em}
} 

\newcommand{\ProductLabel}[1]{ \hspace*{0.5em} {\color{purp}Z_{#1}} \hspace*{6.5em} {\color{gold!75!black}W_{#1}} }

\begin{figure}[h!]
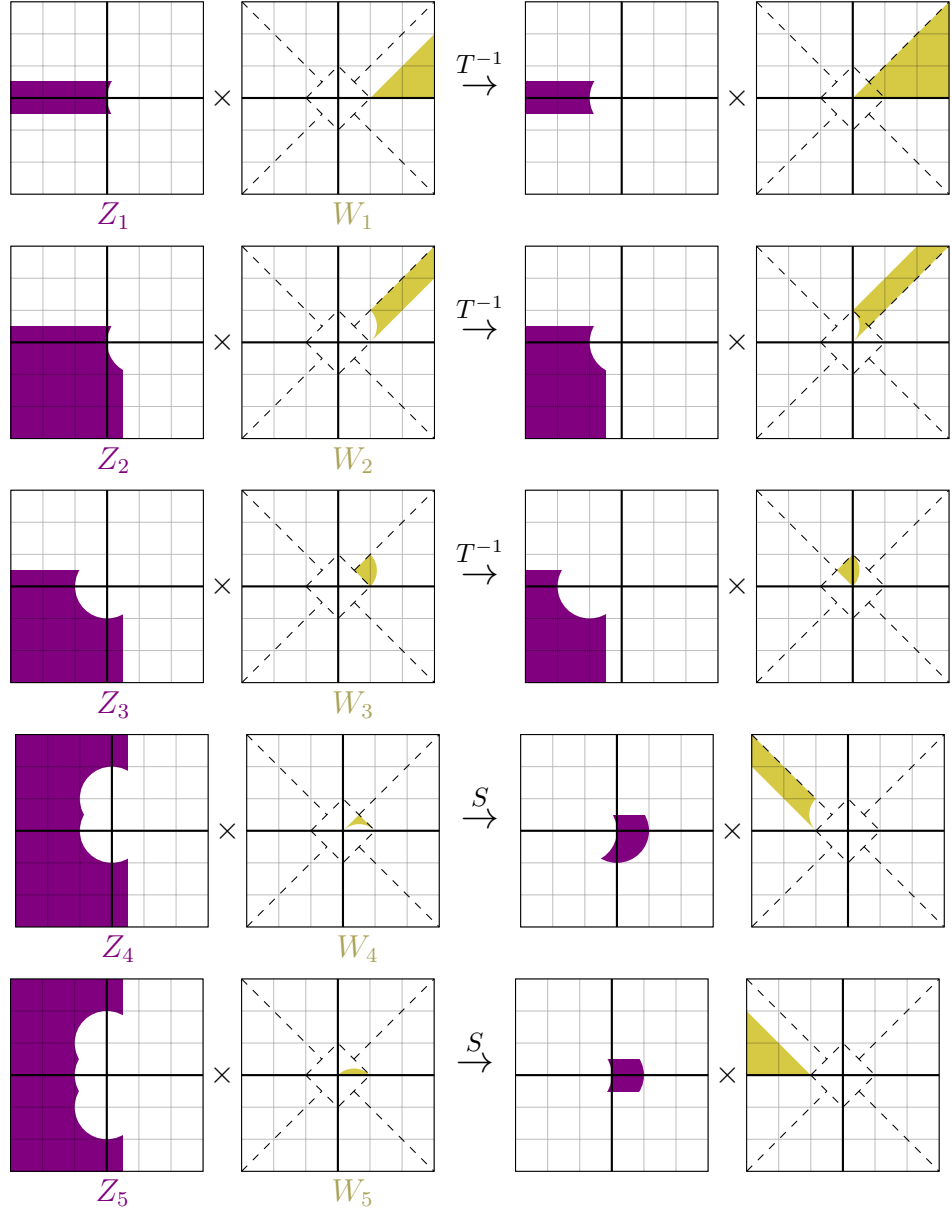

\begin{center}

\Product{\ProductLabel1}{ \fill [purp] (-\Inf,-1/2) rectangle (1/2,1/2); \fill [white] (1,0) circle (1); }{ \fill [gold] (\Inf+1,\Inf) -- (1,0) -- (\Inf+1,0); }
\raisebox{4.5em}{\large$\stackrel{T^{-1}}{\to}$}
\Product{}{ \fill [purp] (-\Inf,-1/2) rectangle (-1/2,1/2); \fill [white] (0,0) circle (1); }{ \fill [gold] (\Inf,\Inf) -- (0,0) -- (\Inf,0); } \\[1.25em]

\Product{\ProductLabel2}{ \fill [purp] (-\Inf,-\Inf) rectangle (1/2,1/2); \fill [white] (1,0) circle (1); }{ \fill [gold] (\Inf+1,\Inf) -- (1,0) arc (-45:45:\sval) -- (\Inf,\Inf); }
\raisebox{4.5em}{\large$\stackrel{T^{-1}}{\to}$}
\Product{}{ \fill [purp] (-\Inf,-\Inf) rectangle (-1/2,1/2); \fill [white] (0,0) circle (1); }{ \fill [gold] (\Inf,\Inf) -- (0,0) arc (-45:45:\sval) -- (\Inf-1,\Inf); } \\[1.25em]

\Product{\ProductLabel3}{ \fill [purp] (-\Inf,-\Inf) rectangle (1/2,1/2); \fill [white] (0,0) circle (1); }{ \fill [gold] (1,0) arc (-45:45:\sval) -- (1/2,1/2) -- cycle; }
\raisebox{4.5em}{\large$\stackrel{T^{-1}}{\to}$}
\Product{}{ \fill [purp] (-\Inf,-\Inf) rectangle (-1/2,1/2); \fill [white] (-1,0) circle (1); }{ \fill [gold] (0,0) arc (-45:45:\sval) -- (-1/2,1/2) -- cycle; } \\[1.25em]

\Product{\ProductLabel4}{ \fill [purp] (-\Inf,-\Inf) rectangle (1/2,\Inf); \fill [white] (0,1) circle (1); \fill [white] (0,0) circle (1); }{ \fill [gold] (1,0) arc (45:135:\sval) -- (1/2,1/2) -- cycle; }
\raisebox{4.5em}{\large$\stackrel{S}{\to}$}
\Product{}{ \PB{0,0} \fill [white] (-1,0) circle (1); \fill [white] (-\Inf,1/2) rectangle (\Inf,\Inf); }{ \fill [xscale=-1,gold] (\Inf+1,\Inf) -- (1,0) arc (-45:45:\sval) -- (\Inf,\Inf); } \\[1.25em]

\Product{\ProductLabel5}{ \fill [purp] (-\Inf,-\Inf) rectangle (1/2,\Inf); \fill [white] (0,1) circle (1); \fill [white] (0,0) circle (1); \fill [white] (0,-1) circle (1); }{ \fill [gold] (1,0) arc (45:135:\sval) -- cycle; }
\raisebox{4.5em}{\large$\stackrel{S}{\to}$}
\Product{}{ \begin{scope} \clip (-1/2,-1/2) rectangle (\Inf,1/2); \PB{0,0} \end{scope} \fill [white] (-1,0) circle (1); }{ \fill [gold] (-\Inf-1,\Inf) -- (-1,0) -- (-\Inf-1,0); }

\end{center}
\caption{The products $Z_k \!\times\! W_k$ and their images under $F\NE$}
\label{fig DNE} 
\end{figure}

As proved in Lemma~\ref{finiteness}, property (\ref{partition image prop}) of Lemma~\ref{partition finiteness} implies that $F\NE(D\NE)$ must have finite product structure with exact same sets $W_k$. Thus there must exist sets $\hat{Z}_k$ such that
\[ F\NE(D\NE) = \bigcup_{k=1}^{40}(\hat{Z}_k\times W_k) = \bigcup_{k=1}^5 \bigcup_{\xi \in \Dih} \xi(\hat{Z}_k \times W_k). \]
Now proving that $F\NE(D\NE) = D\NE$ is equivalent to proving that $\hat{Z}_k = Z_k$ for $k=1,\ldots,5$.

Notice that
$ \hat{Z}_k \times W_k = F\NE(D\NE) \cap (\C \times W_k) $,
so we can describe each $\hat{Z}_k$ as
\< \label{Zk'} \hat{Z}_k = \setform{ z \in \C }{ \exists~ w \in W_k \text{ such that}\, (z,w) \in F\NE(D\NE) }. \>

Let us calculate $\hat{Z}_1$ explicitly here. Based on equation~(\ref{fNE Wk}) and seen in the top and bottom rows of Figure~\ref{fig DNE}, the images $f\NE(W_1)$ and $f\NE(W_5)$ contain some set $\xi W_1, \xi \in \Dih$. Therefore finding the expression for $\hat{Z}_1$ based on (\ref{Zk'}) will involve the images
\begin{align*}
	F\NE(Z_1 \times W_1) &= T^{-1}Z_1 \times T^{-1}W_1 \\
	&= T^{-1}Z_1 \times (W_1 \cup W_2 \cup W_3 \cup W_4 \cup W_5)
\end{align*}
and
\begin{align*}
	F\NE(Z_5 \times W_5) &= SZ_5 \times SW_5 
	= SZ_5 \times \hflip W_1.
\end{align*}
Since $W_1 \subset T^{-1}W_1$, we have that
\[
	T^{-1}Z_1 \times W_1 \subset F\NE(Z_1 \times W_1)
\]
and therefore
\< \label{Z1' from W1}
	T^{-1}Z_1 \subset \hat{Z}_1.
\>
Based on the symmetry of $D\NE$, the sets $\xi(SZ_5 \times \hflip W_1)$ for all $\xi \in \Dih$ are also contained in $F\NE(D\NE)$. Using $\xi = (\hflip)^{-1} = \hflip$ gives that
\[
	\hflip^{-1}(SZ_5 \times \hflip W_1) =
	\hflip SZ_5 \times W_1 \subset F\NE(Z_5 \times Z_5)
\]
and therefore
\< \label{Z1' from W5}
	\hflip SZ_5 \subset \hat{Z}_1.
\>
Since the images of $W_2$, $W_3$, and $W_4$ do not contain any $\xi W_1$, there are no more subsets of $\hat{Z}_1$ to find. Equations~(\ref{Z1' from W1}) and~(\ref{Z1' from W5}) therefore fully describe $\hat{Z}_1$, which can now be given as
\[ \label{Z1'}
	\hat{Z}_1 = T^{-1}Z_1 \cup \hflip SZ_5.
\]
Figure~\ref{fig Z1'} shows that $T^{-1}Z_1 \cup \hflip SZ_5$ is exactly equal to $Z_1$. Thus $\hat{Z}_1 = Z_1$ as desired.

\providecommand{\DrawPlane}{}
\renewcommand{\DrawPlane}[3][scale=0.8]{ 
	\begin{tikzpicture}[#1]
		\begin{scope} \clip (-\Inf,-\Inf) rectangle (\Inf,\Inf); #3 \end{scope}
		\draw [black, thick] (-\Inf,0) -- (\Inf,0);
		\draw [black, thick] (0,-\Inf) -- (0,\Inf); 
		\draw [black,opacity=.25] (-\Inf,-\Inf) grid (\Inf,\Inf);
		\draw (-\Inf,-\Inf) rectangle (\Inf,\Inf);
		\draw (0,-\Inf) node [below] {#2};
	\end{tikzpicture}
}

\begin{figure}[h]
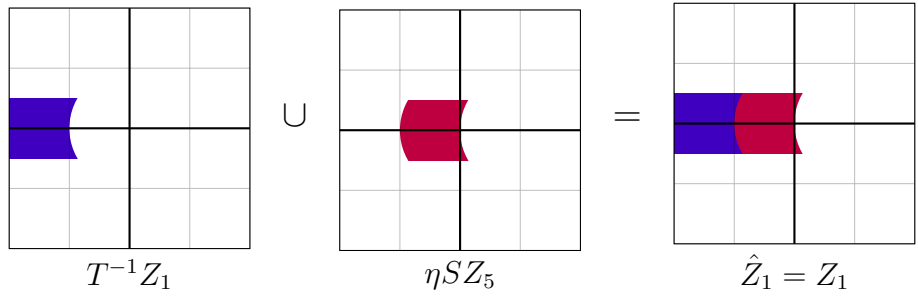

	\begin{center}
	\def\Inf{2}
	\DrawPlane{$T^{-1}Z_1$}{
		\fill [purp!50!blue] (-\Inf,-1/2) rectangle (-1/2,1/2);
		\fill [white] (0,0) circle (1);
	}
	\raisebox{5.6em}{\large$\;\cup\;$}
	\DrawPlane{$\hflip SZ_5$}{
		\begin{scope}
			\clip (-\Inf,-1/2) rectangle (1/2,1/2);
			\fill [purp!50!red] (0,0) circle (1);
		\end{scope}
		\fill [white] (1,0) circle (1);
	}
	\raisebox{5.6em}{\large$\;=\;$}
	\DrawPlane{$\hat{Z}_1 = Z_1$}{
		\begin{scope}
			\clip (-\Inf,-1/2) rectangle (1/2,1/2);
			\fill [purp!50!blue] (-\Inf,-1) rectangle (0,1);
			\fill [purp!50!red] (0,0) circle (1);
		\end{scope}
		\fill [white] (1,0) circle (1);
	}
	\end{center}\vspace*{-1em}
\caption{$\hat{Z}_1 = Z_1$ as a union of two pieces}
\label{fig Z1'} 
\end{figure}

To calculate $\hat{Z}_2$ explicitly, we see\footnote{\,This can be literally seen in the first, second, and fourth rows of Figure~\ref{fig DNE}.} in equation~(\ref{fNE Wk}) that the images \[ f\NE(W_1) = T^{-1}W_1, \qquad f\NE(W_2) = T^{-1}W_2, \qquad f\NE(W_4) = SW_4 \] all contain some set $\xi W_2$ and construct $\hat{Z}_2$ by piecing together the \mbox{appropriate} symmetric copies of $T^{-1}Z_1$, $T^{-1}Z_2$, and $SZ_4$.
This gives
\[
	\begin{array}{l@{\;}p{2.18in}@{}l}
		\hat{Z}_2 
		&= $T^{-1}Z_1 \cup \dflip T^{-1}Z_2 \cup \hflip SZ_4$,
		&\phantom{=\;Z_2}
	\end{array}
\]
and indeed $\hat{Z}_2 = Z_2$.
Similarly, one can show
\[
	\begin{array}{l@{\;}p{2.18in}@{}l}
		\hat{Z}_3 
		&= $T^{-1}Z_1 \cup \dflip T^{-1}Z_2$
		&= Z_3, \\[.1em]
		\hat{Z}_4 
		&= $T^{-1}Z_1 \cup \dflip T^{-1}Z_2 \cup \rot^3 T^{-1}Z_3$
		&= Z_4, \\[.1em]
		\hat{Z}_5 
		&= $T^{-1}Z_1 \cup \rot^3 T^{-1}Z_3 \cup \dflip T^{-1}Z_3$
		&= Z_5.
	\end{array}
\]
Note the expression for $\hat{Z}_5$ correctly includes two different symmetric copies of $T^{-1}Z_3$ because the expression for $T^{-1}W_3$ in (\ref{fNE Wk}) contains two symmetric copies of~$W_5$. Here $\rot^3:z \mapsto -i\,z$ and $\rho:z \mapsto i\,\overline{z}$.

The expressions above show that $\hat{Z}_k = Z_k$ for $k=1,\ldots,5$, which proves that \mbox{$F\NE(D\NE) = D\NE$}. Moreover, the unions in these expressions are disjoint except for the boundaries of the pieces, so the map $F\NE$ is \mbox{essentially} bijective on $D\NE$. This completes the proof of Theorem~\ref{MAIN}.

\section{Trapping region}\label{sec Trapping}

If the orbit of a point $(z,w) \in \CCnoD$ under $F\NE$ enters the set $D\NE$ it will never leave, but do all orbits enter $D\NE$? Unfortunately, this is not currently known, but we now exhibit a set $\Psi\NE \supset D\NE$ that does ``trap points.''

\providecommand{\PsiZ}{-- redefined on the next line --}
\renewcommand{\PsiZ}{A}
\begin{thm} \label{PSI}
	Let $W_1,\ldots,W_5$ be the sets defined in (\ref{Wk}).
	There exist sets $\PsiZ_1,\ldots,\PsiZ_5 \subset \C$ such that the set
	\[ \Psi\NE = \bigcup_{k=1}^5 \bigcup_{\xi\in\Dih} \xi(\PsiZ_k \times W_k) \]
	has the following properties:
	\begin{enumerate}[(i)]
		\item\label{PSI trap prop} for each $(z,w) \in \CCnoD$ with $w$ irrational, there exists an integer $N\ge0$ such that $F\NE^N(z,w) \in \Psi\NE$;
		\item\label{PSI inv prop} $F\NE(\Psi\NE) \subset \Psi\NE$;
		\item\label{PSI set prop} each $\PsiZ_k$ can be explicitly described as a countable union of unit disks and half-planes;
		\item\label{PSI super prop} $D\NE \subset \Psi\NE$.
	\end{enumerate}
\end{thm}
Properties (\ref{PSI trap prop}) and (\ref{PSI inv prop}) together mean that $\Psi\NE$ is a \textdef{trapping region}. The exclusion of rational $w$-values in property (\ref{PSI trap prop}) is because those points have finite orbits under $F\NE$. Also note that property~(\ref{PSI super prop}) is equivalent to $Z_k \subset \PsiZ_k$ for $k=1,\ldots,5$.

Before attempting to describe the sets $\PsiZ_k$, we first look at a set $V\NE$ satisfying only (\ref{PSI trap prop}).
\begin{lem}\label{V lemma}
	Define
	\<[V NE] V\NE := \overline{\D} \times S(\Phi\NE). \>
	For all $(z,w)$ with $w$ irrational, there exists an integer $N\ge0$ such that $F\NE^N(z,w) \in V\NE$.
\end{lem}

\begin{proof}
Since $w$ is irrational, its continued fraction expansion is infinite. Let
\[ w = \cf{ n_0+m_0i,n_1+m_1i,\ldots } = (n_0+m_0i) - \dfrac1{(n_1+m_1i) - \dfrac1\ddots} \]
be the continued fraction for $w$ generated by the diamond algorithm. That is, 
\[ w_0 = w, \qquad n_k+m_ki = c\NE(w_k), \qquad w_k = \frac{-1}{w_k-(n_k+m_ki)}, \]
where $c\NE := c_{\Phi\NE}$ is the diamond choice defined as described in Lemma~\ref{Phi to c}. Equivalently,  $c\NE(w_k) = w_k - f\NE^j(w_k)$ where $j$ is the smallest natural number for which $f\NE^j(w_k) \in \Phi\NE$.

In terms of the functions $T,T^{-1},U,U^{-1}$, and $S$, we can write
\[
	w_{k+1} = S U^{-m_k} T^{-n_k} \cdots S U^{-m_1} T^{-n_1} S U^{-m_0} T^{-n_0} w.
\]
Construct the sequence $\{(z_k,w_k)\}$ with $w_k$ as above and 
\[
	z_{k+1} := S U^{-m_k} T^{-n_k} \cdots S U^{-m_1} T^{-n_1} S U^{-m_0} T^{-n_0} z.
\]
Note that the points $\{(z_k,w_k)\}$ are all in the orbit of $(z_0,w_0) = (z,w)$ under $F\NE$, specifically those points for which $F\NE$ acted by $S(z,w) = (-1/z,-1/w)$ in its last iteration. 
Since inversion was the last operation applied, it must be that $S^{-1} w_{k+1} \in \Phi\NE$ and thus $w_{k+1} \in S(\Phi\NE)$.

Let $\frac{p_k}{q_k}$ be the $k^\text{th}$ convergents for $w$. Then
\begin{align*}
	z &= T^{n_0} U^{m_0} S T^{n_1} U^{m_1} S \cdots T^{n_k} U^{m_k} S(z_{k+1}) 
	= \frac{p_k z_{k+1} - p_{k-1}}{q_k z_{k+1} - q_{k-1}}, \\
	w &= T^{n_0} U^{m_0} S T^{n_1} U^{m_1} S \cdots T^{n_k} U^{m_k} S(w_{k+1})
	= \frac{p_k w_{k+1} - p_{k-1}}{q_k w_{k+1} - q_{k-1}}.
\end{align*}
Hence
\[ z_{k+1} = \frac{q_{k-1} z - p_{k-1}}{q_k z - p_k} = \frac{q_{k-1}}{q_k} + \frac1{q_k^2(\frac{p_k}{q_k} - z)}. \]
Setting $\epsilon_k := q_k^{-2}(\frac{p_k}{q_k} - z)^{-1}$, we have
\begin{align*}
	z_{k+1} &= \frac{q_{k-1}}{q_k} + \epsilon_k \\
	\abs{z_{k+1}} &\le \frac{\abs{q_{k-1}}}{\abs{q_k}} + \abs{\epsilon_k}.
\end{align*}
Since $\frac{p_k}{q_k} \to w \ne z$ by Theorem~\ref{complex convergence}, we have $\abs{\epsilon_k} < \frac1{\abs{q_k}^2}$.
Thus
\[ \abs{z_{k+1}} < \frac{\abs{q_{k-1}}}{\abs{q_k}} + \frac1{\abs{q_k}^2}%
= \frac{ \abs{q_{k-1}}\abs{q_k} + 1 }{ \abs{q_k}^2 }%
. \]
Since $\abs{q_k} \to \infty$ by Theorem~\ref{complex convergence}, there must be an infinite subsequence $\abs{q_{k_j}}$ that is strictly increasing. Thus there are infinitely many $k$ for which $\frac{\abs{q_{k-1}}}{\abs{q_k}} < 1$.
\renewcommand{\ell}{k}%
Thus there exists $\ell$ such that $\abs{z_\ell} \le 1$
, equivalently, $z_\ell \in \overline{\D}$, and we already know that $w_\ell \in S(\Phi\NE)$. This means precisely that $(z_\ell,w_\ell) \in V\NE = \overline{\D} \times S(\Phi\NE)$.
\end{proof}

Although the orbit of every irrational point enters $V\NE$ by Lemma~\ref{V lemma}, orbits will quickly leave $V\NE$ as well. One way to create a trapping region is to take unions of the images of $V\NE$. That is, the set
\<[Psi NE] \Psi\NE := \bigcup_{n\ge0} F\NE^n(V\NE) \>
will be a trapping region by construction.

Equation~(\ref{Psi NE}) actually serves as the definition of $\Psi\NE$, and it remains to show that this set has all the properties claimed in Theorem~\ref{PSI}. The existence of some sets $\PsiZ_k$ satisfying $\Psi\NE = \bigcup_{\,k=1}^{\,5} \PsiZ_k \times W_k$ mod $\Dih$ follows immediately from the finite product structure of $V\NE$; the finite trapping property (\ref{PSI trap prop}) follows from $V\NE \subset \Psi\NE$; and property (\ref{PSI inv prop}), $F\NE(\Psi\NE) \subset \Psi\NE$, holds by construction. Thus it is really the explicit description of the sets, property (\ref{PSI set prop}), that is noteworthy.

\providecommand{\Ball}[1]{} \renewcommand{\Ball}[1]{\overline{B}(#1)}
\begin{proof}[Proof of $($\ref{PSI set prop}$)$]
Modulo $\Dih$, we have $S(\Phi\NE) = W_1 \cup W_2$. That is, the actual equality is $S(\Phi\NE) = \bigcup_{\xi\in\Dih}\xi(W_1\cup W_2)$, but any calculations involving $S(\Phi\NE)$ can be done using only $W_1$ and $W_2$. The remaining calculations are done mod $\Dih$; we start with
\[ V\NE = \overline{\D} \times W_1 \cup \overline{\D} \times W_2. \]
Then
\begin{align*}
	F\NE(\D\times W_1) &= \Ball{-1} \times (W_1 \cup W_2 \cup W_3 \cup W_4 \cup W_5) \end{align*} and \begin{align*} 
	F\NE(\D\times W_2) &= \Ball{-1} \times (\dflip W_2 \cup \dflip W_3 \cup \dflip W_4),
\end{align*}
where $\Ball{c}$ is the unit ball in $\C$ with center $c$ and $\dflip \in \mathrm{Dih}_4$ is the reflection $\dflip: w \mapsto i\,\overline{w}$.
Since $\dflip^{-1} \Ball{-1} = \dflip \Ball{-1} = \Ball{-i}$, we have
\[
	F\NE(V\NE) = \Ball{-1} \times W_1 \cup \big[ \Ball{-1} \cup \Ball{-i} \big] \times [W_2 \cup W_3 \cup W_4].
\]
The image $F\NE^2(V\NE)$ will include half-planes because the image of the disk $\Ball{-1}$ under $S:z\mapsto -1/z$ is the half-plane $\setform{ z }{ \Re\,z > \tfrac12 }$. Actually, $S(W_4) = \eta W_2$ and $S(W_5) = \eta W_1$ with $\eta:(x,y) \mapsto (-x,y)$, so we end up with the half-plane $\setform{ z }{ \Re\,z < \tfrac{-1}2 }$ after flipping the sets back around (the product $\setform{ z }{ \Re\,z > \tfrac12 } \times \eta W_2$ is equivalent mod $\Dih$ to $\setform{ z }{ \Re\,z < \tfrac{-1}2 } \times W_2$, and likewise for a product with $\eta W_1$).
\def\vsp{0.75em}
\begin{align*}
	F\NE\big( \Ball{-1} \times W_1 \big)
	&= \Ball{-2} \times (W_1 \cup W_2 \cup W_3 \cup W_4 \cup W_5) \\[\vsp]
	F\NE\big( [\Ball{-1}\cup \Ball{-i}] \times W_2 \big)
	&= [\Ball{-2} \cup \Ball{-1-i}] \times (\dflip W_2 \cup \dflip W_3 \cup \dflip W_4) \\
	&= [\Ball{-2i} \cup \Ball{-1-i}] \times (W_2 \cup W_3 \cup W_4) \\[\vsp]
	F\NE\big( [\Ball{-1}\cup \Ball{-i}] \times W_3 \big)
	&= [\Ball{-2} \cup \Ball{-1-i}] \times (\dflip W_5 \cup i W_5 \cup i W_4) \\
	&= [\Ball{-2} \cup \Ball{-1-i}] \times \dflip W_5 \\&\qquad\cup\; [\Ball{-2} \cup \Ball{-1-i}]  \times i(W_5 \cup W_4) \\
	&= [\Ball{-2i} \cup \Ball{-1-i}] \times W_5 \\&\qquad\cup\; [\Ball{2i} \cup \Ball{-1+i}]  \times (W_5 \cup W_4) \\[\vsp]
	F\NE\big( [\Ball{-1}\cup \Ball{-i}] \times W_4 \big)
	&= [\setform{z}{\Re\,z > \tfrac12} \cup \setform{z}{\Im\,z < \tfrac{-1}2}] \times \eta W_2 \\
	&= [\setform{z}{\Re\,z < \tfrac{-1}2} \cup \setform{z}{\Im\,z < \tfrac{-1}2}] \times W_2 \\[\vsp]
	F\NE\big( \Ball{-1} \times W_5 \big)
	&= \setform{z}{\Re\,z > \tfrac12} \times \eta W_1 \\
	&= \setform{z}{\Re\,z < \tfrac{-1}2} \times W_1.
\end{align*}
Now we can collect ``like terms'' (products with the same $W_k$) to get
\begin{align*}
	F\NE^2(V\NE)
	&= \Big[ \Ball{-2} \cup \setform{z}{\Re\,z < \tfrac{-1}2} \Big] \times W_1 \\
	&\quad\cup\;\Big[ \Ball{-2} \cup \Ball{-2i} \cup \Ball{-1-i} \\&\qquad\qquad\cup \setform{z}{\Re\,z < \tfrac{-1}2} \cup \setform{z}{\Im\,z < \tfrac{-1}2} \Big] \times W_2 \\
	&\quad\cup\;\Big[ \Ball{-2} \cup \Ball{-2i} \cup \Ball{-1-i} \Big] \times W_3 \\
	&\quad\cup\;\Big[ \Ball{-2} \cup \Ball{-2i} \cup \Ball{-1-i} \cup \Ball{2i} \cup \Ball{-1+i} \Big] \times W_4 \\
	&\quad\cup\;\Big[ \Ball{-2} \cup \Ball{-2i} \cup \Ball{-1-i} \cup \Ball{2i} \cup \Ball{-1+i} \Big] \times W_5.
\end{align*}

\providecommand{\DrawPlane}{}
\renewcommand{\DrawPlane}[3][scale=0.45]{
	\begin{tikzpicture}[#1]
		\begin{scope} \clip (-\Inf,-\Inf) rectangle (\Inf,\Inf); #3 \end{scope}
		\draw [black, thick] (-\Inf,0) -- (\Inf,0);
		\draw [black, thick] (0,-\Inf) -- (0,\Inf); 
		\draw [black,opacity=.25] (-\Inf,-\Inf) grid (\Inf,\Inf);
		\draw (-\Inf,-\Inf) rectangle (\Inf,\Inf);
		\draw (\Inf,0) node [right] {$\!\times {\color{gold!75!black}W_{#2}}$};
	\end{tikzpicture}%
}

\def\Inf{3.25}
\tikzstyle{transpurp}=[draw=purp!75!black,fill=purp,fill opacity=0.5]
\begin{figure}
\begin{center}
	\DrawPlane1{
    	\draw [transpurp] (0,0) circle (1); 
    	\draw [transpurp] (-1,0) circle (1); 
    	\draw [transpurp] (-2,0) circle (1);
    	\draw [transpurp] (-\Inf-1,-\Inf-1) rectangle (-1/2,\Inf+1);
    }
	\DrawPlane2{
    	\draw [transpurp] (0,0) circle (1); 
    	\draw [transpurp] (-1,0) circle (1); 
    	\draw [transpurp] (0,-1) circle (1); 
    	\draw [transpurp] (-2,0) circle (1);
    	\draw [transpurp] (0,-2) circle (1);
    	\draw [transpurp] (-1,-1) circle (1);
    	\draw [transpurp] (-\Inf-1,-\Inf-1) rectangle (-1/2,\Inf+1);
    	\draw [transpurp] (-\Inf-1,-\Inf-1) rectangle (\Inf+1,-1/2);
    }
    \DrawPlane3{
    	\draw [transpurp] (-1,0) circle (1); 
    	\draw [transpurp] (0,-1) circle (1); 
    	\draw [transpurp] (-2,0) circle (1);
    	\draw [transpurp] (0,-2) circle (1);
    	\draw [transpurp] (-1,-1) circle (1);
    } \\[1em]
    \DrawPlane4{
    	\draw [transpurp] (-1,0) circle (1); 
    	\draw [transpurp] (0,-1) circle (1); 
    	\draw [transpurp] (-2,0) circle (1);
    	\draw [transpurp] (0,-2) circle (1);
    	\draw [transpurp] (-1,-1) circle (1);
    	\draw [transpurp] (0,2) circle (1);
    	\draw [transpurp] (-1,1) circle (1);
    }
    \DrawPlane5{
    	\draw [transpurp] (-1,0) circle (1); 
    	\draw [transpurp] (-2,0) circle (1);
    	\draw [transpurp] (0,-2) circle (1);
    	\draw [transpurp] (-1,-1) circle (1);
    	\draw [transpurp] (0,2) circle (1);
    	\draw [transpurp] (-1,1) circle (1);
    }
\end{center}
\vspace*{-0.5em}
\caption{Products in the set $V\NE \cup F\NE(V\NE) \cup F\NE^2(V\NE)$}
\label{fig Psi after 2} 
\end{figure}

Using inclusions such as $\Ball{-2} \subset \setform{z}{\Re\,z < \tfrac{-1}2}$, the expression of $F\NE^2(V\NE)$ as a union of products can be reduced somewhat, and the expression of $V\NE \cup F\NE(V\NE) \cup F\NE^2(V\NE)$ can be reduced significantly.

\begin{quote}\textbf{Remark.} Geometrically, the ``outer layer'' of disks (furthest centers from origin in the $1$-norm) in the products with $W_3,W_4,W_5$ in Figure~\ref{fig Psi after 2} continue to move away from the origin with higher iterations of $F_\diamond$.\end{quote}

This process continues with $F\NE^3(V\NE)$, $F\NE^4(V\NE)$, and $F\NE^5(V\NE)$, at which one finds that
\( \bigcup_{n=0}^5 F\NE^n(V\NE) = \bigcup_{n=0}^4 F\NE^n(V\NE), \)
so we can stop iterating and terminate with
\( \Psi\NE = \bigcup_{\,n=0}^{\,4} F\NE^n(V\NE). \)

A graphical depiction of $\Psi\NE$ can be found in Figure~\ref{fig Psi}. In formulas, the sets are as follows:
\begin{align}
	\PsiZ_1 &= \Ball{0} \cup \Ball{-1+ i} \cup \Ball{-1- i} \cup \setform{ z }{ \Re\,z \le \tfrac{-1}2 } \notag \\
	\PsiZ_2 &= \Ball{0} \cup \Ball{-1+i} \cup \setform{ z }{ \Re\,z \le \tfrac{-1}2 } \cup \setform{ z }{ \Im\,z \le \tfrac{-1}2 } \notag \\
	\PsiZ_3 &= \Ball{-1} \cup \Ball{-i} \cup \Ball{1-2i} \notag\\&\qquad \cup \setform{ z }{ \Re\,z \le \tfrac{-1}2 } \cup \setform{ z }{ \Im\,z \le \tfrac{-3}2 } \label{PsiZ k} \\
	\PsiZ_4 &= \Ball{2i} \cup \Ball{-1+i} \cup \Ball{-1} \cup \Ball{-i} \cup \Ball{1-2i} \notag\\&\qquad \cup \setform{ z }{ \Re\,z \le \tfrac{-1}2 } \cup \setform{ z }{ \abs{\Im\,z} \ge \tfrac{3}2 } \notag \\
	\PsiZ_5 &= \Ball{2i} \cup \Ball{-2i} \cup \Ball{-1+ i} \cup \Ball{-1- i} \cup \Ball{-1} \notag\\&\qquad \cup \setform{ z }{ \Re\,z \le \tfrac{-1}2 } \cup \setform{ z }{ \abs{\Im\,z} \ge \tfrac{3}2 } \notag
\end{align}
In equation~(\ref{PsiZ k}), each $\PsiZ_k$ is a countable union of unit disks and half-planes, so these descriptions agree with property~(\ref{PSI set prop}) of Theorem~\ref{PSI}.
\end{proof}

Figure~\ref{fig Psi} shows the set $\Psi\NE$ and the bijectivity domain $D\NE$ calculated previously. As usual, the set $\Psi\NE$ is really $\bigcup_{\,k=1}^{\,5} \bigcup_{\,\xi\in\mathrm{Dih}_4} (\xi \PsiZ_k \times \xi W_k)$ even though only $\PsiZ_1\times W_1, \ldots, \PsiZ_5\times W_5$ are shown in the figure. From these pictures it is clear that $Z_k \subset \PsiZ_k$ for all $k$ (this can also be verified using equations~(\ref{Zk}) and~(\ref{PsiZ k})), which implies property (\ref{PSI super prop}), $D\NE \subset \Psi\NE$. This completes the proof of Theorem~\ref{PSI} in its entirety.

\renewcommand{\Product}[3][scale=1.8]{
\begin{tikzpicture}[#1]
	\begin{scope}[scale=1/\Inf]
		\begin{scope} \clip (-\Inf,-\Inf) rectangle (\Inf,\Inf); #2 \end{scope}
		
		\draw [black, thick] (-\Inf,0) -- (\Inf,0);
		\draw [black, thick] (0,-\Inf) -- (0,\Inf);
		 
		\draw [black,opacity=.25] (-\Inf,-\Inf) grid (\Inf,\Inf);
		\draw (-\Inf,-\Inf) rectangle (\Inf,\Inf);
	\end{scope}
		
	\draw (1.25,0) node {$\times$};
	
	\def\tempInf{\Inf}
	\def\Inf{2}
	\begin{scope}[xshift=2.5cm]
	\begin{scope}[scale=1/\Inf]
		\begin{scope} \clip (-\Inf,-\Inf) rectangle (\Inf,\Inf); #3 \end{scope}
		
		\draw [black, thick] (-\Inf,0) -- (\Inf,0);
		\draw [black, thick] (0,-\Inf) -- (0,\Inf);
		 
		\draw [black,opacity=.25] (-\Inf,-\Inf) grid (\Inf,\Inf);
		\draw [dashed] (-\Inf,-\Inf) -- (-1/2,-1/2);
		\draw [dashed] (1/2,1/2) -- (\Inf,\Inf);
		\draw [dashed] (-\Inf,\Inf) -- (-1/2,1/2);
		\draw [dashed] (1/2,-1/2) -- (\Inf,-\Inf);
		\draw [dashed] (-1,0) -- (0,1) -- (1,0) -- (0,-1) -- cycle;
		\draw (-\Inf,-\Inf) rectangle (\Inf,\Inf);
	\end{scope}
	\end{scope}
	\def\Inf{\tempInf}
\end{tikzpicture}
} 

\def\Inf{3.5}
\tikzstyle{Dpurp}=[purp!67!black]
\tikzstyle{PSIpurp}=[purp!67!white]
\begin{figure}[h]
\begin{center}
\Product{
	\fill [PSIpurp] (0,0) circle (1);
	\fill [PSIpurp] (-1,1) circle (1);
	\fill [PSIpurp] (-1,-1) circle (1);
	\fill [PSIpurp] (-\Inf,-\Inf) rectangle (-1/2,\Inf);
	\fill [Dpurp] (-\Inf,1/2) -- (0.13397,1/2) arc (150:210:1) -- (-\Inf,-1/2);
}{ \fill [gold] (\Inf+1,\Inf) -- (1,0) -- (\Inf+1,0); }
\\[1em]
\Product{
	\fill [PSIpurp] (0,0) circle (1);
	\fill [PSIpurp] (-1,1) circle (1);
	\fill [PSIpurp] (-1,-1) circle (1);
	\fill [PSIpurp] (-\Inf,-\Inf) rectangle (-1/2,\Inf);
	\fill [PSIpurp] (-\Inf,-\Inf) rectangle (\Inf,-1/2); 
	\fill [Dpurp] (-\Inf,1/2) -- (0.13397,1/2) arc (150:240:1) -- (1/2,-\Inf) -- (-\Inf,-\Inf);
}{ \fill [gold] (\Inf+1,\Inf) -- (1,0) arc (-45:45:\sval) -- (\Inf,\Inf); }
\\[1em]
\Product{
	\fill [PSIpurp] (-1,0) circle (1);
	\fill [PSIpurp] (0,-1) circle (1);
	\fill [PSIpurp] (1,-2) circle (1);
	\fill [PSIpurp] (-\Inf,-\Inf) rectangle (-1/2,\Inf);
	\fill [PSIpurp] (-\Inf,-\Inf) rectangle (\Inf,-3/2); 
	\fill [Dpurp] (-\Inf,1/2) -- (0.13397-1,1/2) arc (150:300:1) -- (1/2,-\Inf) -- (-\Inf,-\Inf);
}{ \fill [gold] (1,0) arc (-45:45:\sval) -- (1/2,1/2) -- cycle; }
\\[1em]
\Product{
	\fill [PSIpurp] (-\Inf,3/2) rectangle (\Inf,\Inf);
	\fill [PSIpurp] (0,2) circle (1);
	\fill [PSIpurp] (-1,1) circle (1);
	\fill [PSIpurp] (-1,0) circle (1);
	\fill [PSIpurp] (-\Inf,-\Inf) rectangle (-1/2,\Inf);
	\fill [PSIpurp] (-\Inf,-\Inf) rectangle (\Inf,-3/2);
	\fill [PSIpurp] (-1,-1) circle (1); 
	\fill [PSIpurp] (0,-2) circle (1); 
	\fill [PSIpurp] (0,-1) circle (1);
	\fill [PSIpurp] (1,-2) circle (1); 
	\fill [Dpurp] (1/2,\Inf) -- (1/2,2-0.13397) arc (60:210:1) arc (150:300:1) -- (1/2,-\Inf) -- (-\Inf,-\Inf) -- (-\Inf,\Inf);
}{ \fill [gold] (1,0) arc (45:135:\sval) -- (1/2,1/2) -- cycle; }
\\[1em]
\Product{
	\fill [PSIpurp] (-\Inf,3/2) rectangle (\Inf,\Inf);
	\fill [PSIpurp] (0,2) circle (1);
	\fill [PSIpurp] (-1,1) circle (1);
	\fill [PSIpurp] (-1,0) circle (1);
	\fill [PSIpurp] (-\Inf,-\Inf) rectangle (-1/2,\Inf);
	\fill [PSIpurp] (-\Inf,-\Inf) rectangle (\Inf,-3/2);
	\fill [PSIpurp] (-1,-1) circle (1);
	\fill [PSIpurp] (0,-2) circle (1); 
	%
	\fill [Dpurp] (1/2,\Inf) -- (1/2,2-0.13397) arc (60:210:1) arc (150:210:1) arc (150:300:1) -- (1/2,-\Inf) -- (-\Inf,-\Inf) -- (-\Inf,\Inf);
}{ \fill [gold] (1,0) arc (45:135:\sval) -- cycle; }
\end{center}
\caption{\mbox{Trapping set $\Psi\NE$ {\color{purp!75!white}(light)} and bijectivity domain $D\NE$ {\color{purp!67!black}(dark)}.}}
\label{fig Psi} 
\end{figure}

\pagebreak[10]

\end{document}